\documentclass[a4paper,10pt]{article}
\usepackage[utf8]{inputenc}
\usepackage[english]{babel}
\usepackage{amsmath,amsthm,amssymb}
\usepackage{url}
\usepackage{eucal}
\usepackage{bm}
\usepackage{fancyhdr}
\usepackage{color}
\usepackage{tikz}
\usepackage{graphicx}
\usepackage[font=small]{caption}

\newtheorem{theorem}{Theorem}
\newtheorem{proposition}[theorem]{Proposition}
\newtheorem{lemma}[theorem]{Lemma}

\theoremstyle{definition}
\newtheorem{definition}[theorem]{Definition}
\newtheorem{remark}[theorem]{Remark}

\numberwithin{equation}{section}

\newcommand{\ddiv}{\operatorname{div}}
\newcommand{\uHt}{\tilde{u}_H}
\newcommand{\smooth}{^{\mathit{reg}}}
\newcommand{\loc}{^{(\ell)}}
\newcommand{\sym}{\operatorname{sym}}
\newcommand{\card}{\operatorname{card}}

\newcommand{\T}{\mathcal{T}}
\newcommand{\Cor}{\mathcal{C}}
\newcommand{\nei}{\mathsf{N}}
\newcommand{\wcba}{\bm{\mathrm{wcba}}}

\begin{document}
\pagestyle{fancy}
\fancyhead{}
\setlength{\headheight}{14pt}
\renewcommand{\headrulewidth}{0pt}

\fancyhead[c]{\small \it Numerical stochastic homogenization}
\title{Numerical stochastic homogenization by quasilocal effective diffusion tensors}

\author{%
       D. Gallistl\thanks{%
                      Department of Applied Mathematics,
                      University of Twente,
                      7500 AE Enschede, The Netherlands}
        \and 
       {D. Peterseim\thanks{Institut f\"ur Mathematik,     
            Universit\"at Ausgburg,
         86159 Augsburg, Germany}}
       }
\date\today

\maketitle

\begin{abstract}
This paper proposes a numerical upscaling procedure for elliptic 
boundary value problems with 
diffusion tensors that vary randomly on small scales.
The method compresses the random partial differential operator to an effective quasilocal deterministic 
operator that represents the expected solution on a coarse
scale of interest.
Error estimates consisting of a~priori and a~posteriori terms are 
provided that allow one to quantify the impact of uncertainty in the 
diffusion coefficient on the expected effective response of the process. 
\end{abstract}

{\small

}

\section{Introduction}
Homogenization is a tool of mathematical modeling to identify 
reduced descriptions of the macroscopic response of multiscale models.
In the context of the prototypical elliptic model problem
\begin{equation*}
-\ddiv (\mathbf{A}_\varepsilon \nabla \mathbf{u}_\varepsilon) = f 
\end{equation*}
microscopic features on some characteristic length scale $\varepsilon$ 
are encoded in the diffusion coefficient $\mathbf{A}_\varepsilon$ and 
homogenization studies the limit as $\varepsilon$ tends to zero. 
It turns out that suitable limits represented by the so-called 
effective or homogenized coefficient exist in fairly general 
settings in the framework of $G$-, $H$-, or two-scale convergence
\cite{Spagnolo:1968,Gio75,Mur78,MR990867,MR1185639}. 
However, the effective coefficient is rarely given explicitly
and even its implicit representation based on cell problems  
 usually requires structural assumptions on the 
sequence of coefficients $\mathbf{A}_\varepsilon$ such as local periodicity 
and scale separation \cite{BLP78,JKO94}.
Moreover, in many interesting applications such as geophysics and material sciences where $\mathbf{A}_\varepsilon$ represents porosity or permeability,
complete explicit knowledge of the coefficient is unlikely.
The coefficient is rather the result of measurements that underlie
errors or it is the result of singular measurements combined with
inverse modeling. In any case, it is very likely that there is
uncertainty in the data $\mathbf{A}_\varepsilon$.
The question is how 
this uncertainty on the fine scale $\varepsilon$ changes the expected
macroscopic response of the process. 
For works on analytical stochastic homogenization we refer to the 
classical works 
\cite{Kozlov1979,PapanicolaouVaradhan1981,Yurinskii1986}
and the recent approaches
\cite{BourgeatPiatnitski2004},
\cite{GloriaOtto2011,GloriaOtto2012,GloriaNeukammOtto2014,GloriaNeukammOtto2015,GloriaOtto2015,DuerinckxGloriaOtto2016,GloriaOtto2017}, and
\cite{ArmstrongSmart2016,ArmstrongKuusiMourrat2017}.
Computational methods in stochastic homogenization
comprise
\cite{EMingZhang2005,ChenSavchuk2007,GintingMalqvistPresho2010,BalJing2011,LeBrisLegollThomines2014,JagalurSahniDoostanOberai2014,ZhangCiHou2015}
and the review article
\cite{AnantharamanEtAl_Review2012}.
However, important practical and theoretical questions remain open. 
Among them are the truncation of cell problems
(see \cite{Gloria2012m2an,Gloria2012proc,GloriaHabibi2016} for recent progress on this aspect),
the role of discretization 
on the different scales and its interaction with the homogenization process
as well as the incorporation of macroscopic boundary conditions.

This paper addresses all these issues to some extent. It provides a method to compute effective, deterministic
models such that the corresponding 
discrete solution is close to $\mathbb{E}[\mathbf{u}_\varepsilon]$, the expected value of 
$\mathbf{u}_\varepsilon$; closeness is meant in the $L^2$ sense, so that
a meaningful approximation is achieved on some coarse 
scale of interest $H$ (observation scale) which is typically linked to some triangulation
$\T_H$ of mesh-size $H$.
The numerical method is based on the multiscale approach of
\cite{MalqvistPeterseim2014}, sometimes referred to as 
Localized Orthogonal Decomposition (LOD), that was developed for the 
deterministic case (see also \cite{HenningPeterseim2013,GallistlPeterseim2015,Peterseim:2015,Kornhuber.Peterseim.Yserentant:2016}).
The basis functions therein are constructed by local corrections 
(generally different from those of analytical homogenization theory)
that solve some elliptic fine-scale problem on localized patch domains.
Their supports are determined by oversampling
lengths $H\lvert\log H\rvert$, where $H$ denotes the mesh-size of a
finite element triangulation $\T_H$ on the observation scale.
This choice of oversampling is justified by the exponential decay of the
correctors away from their source \cite{MalqvistPeterseim2014,HenningPeterseim2013,KornhuberYserentant2015,Kornhuber.Peterseim.Yserentant:2016}. The method leads to quasi-optimal a~priori error estimates and
can dispense with any assumptions on scale separation.

The recent work \cite{GallistlPeterseim2017} gives a re-interpretation
of the method from \cite{MalqvistPeterseim2014} as a quasilocal
discrete integral operator for the deterministic case.
In a further compression step, this representation allows to extract
a piecewise constant diffusion tensor.
An application of this procedure for any atom $\omega$ in the
probability space leads to an integral operator
$\mathcal A_H$ (depending on the stochastic variable)
and a corresponding piecewise constant random field
$\mathbf{A}_H$ on the scale $H$.
It turns out that this viewpoint is useful in the stochastic setting
because it allows to average in the stochastic variable
over effective coefficients rather than over multiscale basis functions
and to thereby characterize the resulting effective model in terms
of quasi-local coefficients and even deterministic PDEs.
The averages are given by
$\bar{\mathcal A}_H:=\mathbb{E}[\mathcal A_H]$  
and $\bar{A}_H:= \mathbb{E}[\mathbf A_H]$
and constitute deterministic models
which we refer to as quasi-local ($\bar{\mathcal A}_H$) and
local ($\bar{A}_H$), respectively.
The proposed method covers the case of bounded polytopes, which 
appears still open in analytical stochastic homogenization.
The method itself can dispense with any a~priori information
on the coefficient. The validity of the discrete model is assessed
via an a~posteriori model error estimator.
In order to make the computation of 
$\bar{\mathcal A}_H$ and $\bar A_H$ feasible,
one can exploit the structure (if available) of the stochastic coefficient
$\mathbf{A}_\varepsilon$ as well as the underlying mesh.
Provided the dependence of the stochastic variable has a suitable structure,
sampling procedures for $\mathbf{A}_H$ are purely local and allow
to restrict the computations to reference configurations.

We provide error estimates for the expected error in the $L^2$ norm as well as the $L^2$ norm of the expected error. The upper bounds are combined from a~priori terms and a~posteriori
terms. The latter contributions are determined by the statistics
of the local fluctuations of the upscaled coefficient.
Numerical evidence for 
an uncorrelated
model coefficient suggests that the error 
estimator is not the dominant part in the error, as long as the usual
scaling $H\approx (\varepsilon/H)^{d/2}$ from the central limit 
theorem (CLT) is satisfied.

\smallskip
The structure of this article is as follows. 
Section~\ref{s:notation} introduces the general model problem,
relevant notation for data structures and function spaces,
and gives an example of a possible model situation.
Section~\ref{s:modelupscaling}
presents the upscaling procedure.
Section~\ref{s:erroranalysis} provides error estimates
in the $L^2$ norm.
Numerical experiments are presented in
Section~\ref{s:num}.
The comments of Section~\ref{s:conclusion} conclude the paper.

\smallskip
Standard notation on Lebesgue and Sobolev spaces
applies throughout this paper. 
The notation $a\lesssim b$ abbreviates $a\leq C b$ for some
constant $C$ that is independent of the mesh-size,
but may depend on the contrast of the coefficient $\mathbf{A}$;
$a\approx b$ abbreviates $a\lesssim b\lesssim a$.
The symmetric part of a quadratic matrix $M$ is
denoted by $\sym(M)$.

\section{Model problem and notation}\label{s:notation}
This section describes the model problem and some notation
on finite element spaces.
Finally, an example of a possible model situation is discussed.

\subsection{Model problem}
Let $(\Omega,\mathcal{F},\mathbb{P})$ be a probability space with set of events $\Omega$, $\sigma$-algebra $\mathcal{F}\subseteq 2^\Omega$ and probability measure $\mathbb{P}$. 
The expectation operator is denoted by $\mathbb{E}$.
Let $D\subseteq\mathbb R^d$ for $d\in\{1,2,3\}$ be
a bounded Lipschitz polytope. The set of admissible coefficients reads
\begin{equation}\label{e:classM}
\mathcal{M}(D,\alpha,\beta) 
= \left\{\begin{aligned}A\in 
  L^\infty(D;\mathbb{R}^{d\times d}_{\mathrm{sym}})&\text{ s.t. }\,\alpha|\xi|^2 \leq (A(x)\xi)\cdot\xi \leq \beta |\xi|^2\\&\text{for a.e. }x\in D\text{ and all }\xi\in\mathbb{R}^d
\end{aligned}
\right\}.
\end{equation}
 Here, $\mathbb{R}^{d\times d}_{\mathrm{sym}}$ denotes
 the set of symmetric $d\times d$ matrices.
Note that the elements of $A\in \mathcal{M}(D,\alpha,\beta)$ are fairly free to vary within the bounds $\alpha$ and $\beta$ and that we do not assume any frequencies of variation or smoothness. 

Let $\mathbf{A}$ be an $\mathcal{M}(D,\alpha,\beta)$-valued, pointwise symmetric random field with $\beta>\alpha>0$ and let, for the sake of readability, 
$f\in L^2(D)$ be deterministic. 
Throughout this article we suppress the 
characteristic length scale $\varepsilon$ of the diffusion coefficient
in the notation and write $\mathbf{A}$ instead of $\mathbf{A}_\varepsilon$.
Consider the model problem
\begin{equation}\label{e:modelstrong}
\left\{
\begin{aligned}
-\ddiv (\mathbf{A}(\omega)(x) \nabla \mathbf{u}(\omega)(x) ) &= f(x),&x\in D\\
\mathbf{u}(\omega)(x)&=0,&x\in \partial D
\end{aligned}
\right\}\quad\text{for almost all }\omega\in\Omega.
\end{equation}
Denote the energy space by
$V:=H^1_0(D)$.
The weak formulation of \eqref{e:modelstrong} seeks a $V$-valued random field $\mathbf{u}$ such that for almost all $\omega\in\Omega$
\begin{equation}\label{e:diff1drandweak}
 \int_D(\mathbf{A}(\omega)(x)\nabla\mathbf{u}(\omega))(x)\cdot\nabla v(x)\,dx = \int_D f(x) v(x)\,dx\quad\text{for all }v\in V.
\end{equation}
The reformulation of this problem in the Hilbert space $L^2(\Omega;V)$ of $V$-valued random fields with finite second moments leads to a coercive variational problem that seeks $\mathbf{u}\in L^2(\Omega;V)$ such that 
\begin{equation*}
 \int_\Omega\int_D(\mathbf{A}(\omega)\nabla\mathbf{u}(\omega)(x))\cdot\nabla\mathbf{v}(\omega)(x)\,dx\,d\mathbb{P}(\omega) = \int_\Omega\int_D f(x) \mathbf{v}(\omega)(x)\,dx\,d\mathbb{P}(\omega)
\end{equation*}
holds for all $\mathbf{v}\in L^2(\Omega;V)$.
It is easily checked that this is a well-posed problem in the sense of the Lax-Milgram theorem with a coercive and bounded bilinear form
$$\begin{aligned}
 a: L^2(\Omega;V)\times L^2(\Omega;V)&\rightarrow \mathbb{R},\\(\mathbf{u},\mathbf{v})&\mapsto \int_\Omega\int_D(\mathbf{A}(\omega)(x)\nabla\mathbf{u}(\omega)(x))\cdot\nabla\mathbf{v}(\omega)(x)\,dx\,d\mathbb{P}(\omega)
\end{aligned}$$
and a bounded linear functional $F$ on $L^2(\Omega;V)$ given by
$$\mathbf{v}\mapsto  \int_\Omega\int_D f(x) \mathbf{v}(\omega)(x)\,dx\,d\mathbb{P}(\omega).$$
This shows that, for any $f\in L^2(D)$, there exists a unique solution 
$\mathbf{u}\in L^2(\Omega;V)$ with
\begin{equation*}
 \|\nabla\mathbf{u}\|_{L^2(\Omega;V)}
:=\left(\int_\Omega\int_D |\nabla(\mathbf{u}(\omega))(x)|^2\,dx \,d\mathbb{P}(\omega)\right)^{1/2}
\leq C(D)\alpha^{-1}\|f\|_{L^2(D)}. 
\end{equation*}
Though it would be possible, we disregard the possibility of more general
$f\in H^{-1}(D)$ or uncertainty in the right-hand side $f$ in this article. 

\begin{remark}
The parameter $\varepsilon$ refers to some scale that resolves the 
stochastic data. We do not assume any particular structure;
the coefficient 
$\mathbf{A}=\mathbf{A}_\varepsilon$ is not necessarily part of
some ergodic sequence. 
Our viewpoint is that of coarsening/reducing the given
model on the fixed scale $\varepsilon$ to the observation scale $H$
rather than that of the asymptotics for small $\varepsilon$.
\end{remark}

\subsection{Finite element spaces}
Let $\T_H$ be a quasi-uniform regular simplicial triangulation of $D$
and let $V_H$ denote the standard $P_1$ finite element space,
that is, the subspace of $V$ consisting of piecewise first-order
polynomials.
Given any subdomain $S\subseteq\overline D$, define its neighbourhood
via
\begin{equation*}
\nei(S):=\operatorname{int}
          \Big(\cup\{T\in\T_H\,:\,T\cap\overline S\neq\emptyset  \}\Big).
\end{equation*}
Furthermore, we introduce for any $m\geq 2$ the patch extensions
\begin{equation*}
\nei^1(S):=\nei(S)
\qquad\text{and}\qquad
\nei^m(S):=\nei(\nei^{m-1}(S)) .
\end{equation*}
Throughout this paper, we assume that the coarse-scale mesh $\T_H$
is quasi-uniform.
The global mesh-size reads 
$H:=\max\{\operatorname{diam}(T):T\in\T_H\}$.
Note that the shape-regularity implies that there is a uniform bound 
$C(m)$
on the number of elements in the $m$th-order patch,
$
\card\{ K\in\T_H\,:\, K\subseteq \overline{\nei^m(T)}\}
\leq C(m)
$
for all ${T\in\T_H}$.
The constant $C(m)$ depends polynomially on $m$.
The set of interior $(d-1)$-dimensional hyper-faces 
of $\T_H$ is denoted by
$\mathcal F_H$. For a piecewise continuous function $\varphi$,
we denote the jump across an interior edge by $[\varphi]_F$,
where the index $F$ will be sometimes omitted for brevity.
The space of piecewise constant functions
(resp.\ $d\times d$ matrix fields)
is denoted by
$P_0(\T_H)$ (resp.\ $P_0(\T_H;\mathbb R^{d\times d})$).

Let $I_H:V\to V_H$ be a 
surjective
quasi-interpolation operator that
acts as an $H^1$-stable and $L^2$-stable
quasi-local projection in the sense that
$I_H\circ I_H = I_H$ and that
for any $T\in\T_H$ and all $v\in V$ there holds
\begin{align*}
H^{-1}\|v-I_H v\|_{L^2(T)} + \|\nabla I_H v \|_{L^2(T)}
&
\leq C_{I_H} \|\nabla v\|_{L^2(\nei(T))} 
\\
\|I_H v\|_{L^2(T)}
&
\leq C_{I_H} \|v\|_{L^2(\nei(T))} .
\end{align*}
Since $I_H$ is a stable projection from $V$ to $V_H$,
any $v\in V$ is quasi-optimally approximated by $I_H v$
in the $L^2(D)$ norm as well as in the $H^1(D)$ norm.
One possible choice
is to define $I_H:=I^{\mathit{c}}_H\circ\Pi_H$, where
$\Pi_H$ is the $L^2(D)$-orthogonal projection onto 
the space $P_1(\T_H)$ of piecewise affine (possibly discontinuous)
functions
and $I^{\mathit{c}}_H$ is the averaging operator that maps $P_1(\T_H)$ to $V_H$ by
assigning to each free vertex the arithmetic mean of the corresponding
function values of the neighbouring cells, that is, for any $v\in P_1(\T_H)$
and any free vertex $z$ of $\T_H$,
\begin{equation*}
(I^{\mathit{c}}_H(v))(z) =
           \sum_{\substack{T\in\T_H\\\text{with }z\in T}}v|_T (z) 
           \bigg/
           \card\{K\in\T_H\,:\,z\in K\}.
\end{equation*}
This choice of $I_H$ is employed in our numerical experiments.

\subsection{Discrete stochastic setting}\label{ss:Teps}
In this subsection we briefly describe one possible discrete stochastic
setting where the uncertainty is encoded by a triangulation
$\T_\varepsilon$.
Although it is not the most general coefficient that can be treated
with the methods described below, it appears as a natural model situation
in a multiscale setting and will therefore be utilized in the 
numerical experiments from Section~\ref{s:num}.

We assume that the triangulation $\T_\varepsilon$
describing the multiscale structure of $\mathbf{A}$ is a uniform
refinement of the triangulation $\T_H$ on the observation scale.
Let $\T_\varepsilon$ denote a uniform triangulation.
The probability space reads
\begin{equation*}
 \Omega = \prod_{T\in\T_\varepsilon} [\alpha,\beta]
        = [\alpha,\beta]^{\card\T_\varepsilon}.
\end{equation*}
Each $\omega=(\omega_T)_{T\in \T_\varepsilon}\in\Omega$
can be identified with a scalar
$\T_\varepsilon$-piecewise constant function 
$\iota\omega$ over $D$ with
$\iota\omega|_T = \omega_T$ for any $T\in\T_\varepsilon$.
The scalar random diffusion coefficient
$\mathbf{A}=\mathbf{A}_\varepsilon$
is a random variable 
$\mathbf{A}\in L^2(\Omega;\mathcal{M}(D,\alpha,\beta))$.
The values  are piecewise constant in space, that is
\begin{equation*}
\mathbf{A}(\bullet,\omega) = \iota\omega\in P_0(\T_\varepsilon)
 \quad\text{for any }\omega\in\Omega .
\end{equation*}
Of course, similar settings are possible for tensor-valued diffusion
coefficients.

\section{Upscaling method}\label{s:modelupscaling}
This section describes
the proposed upscaling methods.

\subsection{Upscaling with a quasi-local effective model}\label{ss:upscalingQloc}
This subsection describes the computation of a quasi-local effective
coefficient. The underlying model does not correspond to a PDE
but rather to a discrete integral operator on finite element
spaces. The method is very flexible in that it is not restricted 
to (quasi-)periodic situations and is able to include boundary
conditions.

The upscaling procedure presented here is based on the multiscale
approach of \cite{MalqvistPeterseim2014,HenningPeterseim2013}. 
For the deterministic case, it was shown in \cite{GallistlPeterseim2017}
that a variant of those methods corresponds to a
finite element system with a quasilocal discrete integral operator.
Its construction for the stochastic setting is described 
in the following.

Let $W:=\operatorname{ker}I_H\subseteq V$ denote the kernel of $I_H$.
The space $W$ is referred to as fine-scale space.
For any element $T\in\T_H$ define the extended element patch
$D_T:=\nei^\ell(T)$ of order $\ell$.
The nonnegative integer $\ell$ is referred to as the
\emph{oversampling parameter}.
As a crucial parameter in the design of the multiscale method,
it is inherent to all quantities in the upscaled model.
The parameter will always be chosen 
$\ell\approx\lvert\log H\rvert$.
For better readability we will suppress the explicit dependence on
$\ell$ whenever there is no risk of confusion, but stress the fact that
quantities like 
$\mathbf{q}_{T,j}$ $\Cor$, $\mathcal A_H$, etc.\
defined below should be understood as
$\mathbf{q}_{T,j}\loc$ $\Cor\loc$, $\mathcal A_H\loc$.

Let $W_{D_T}\subseteq W$ denote the space of functions from $W$
that vanish outside $D_T$.
For any $T\in\T_H$, any $j\in\{1,\dots,d\}$,
and any $v_H\in V_H$,
the function 
$\mathbf{q}_{T,j}\in L^2(\Omega;W_{D_T})$ solves
\begin{equation}\label{e:qTelementcorrELL}
\int_{D_T} \nabla w\cdot(\mathbf{A} \nabla \mathbf{q}_{T,j})\,dx
=
\int_{T} \nabla w \cdot(\mathbf{A} e_j)\,dx
\quad\text{for all }w\in W_{D_T}.
\end{equation}
Here $e_j$ ($j=1,\dots,d$) is the $j$th Cartesian unit vector.
The functions $\mathbf{q}_{T,j}$ are called element correctors.
We emphasize that the element correctors
$\mathbf{q}_{T,j}$
are $W_{D_T}$-valued random variables.
Given $v_H\in V_H$, we define the corrector 
$\Cor v_H\in L^2(\Omega;W)$
by
\begin{equation}\label{e:corexpansionELL}
\Cor  v_H
= \sum_{T\in\T_H} \sum_{j=1}^d (\partial_j v_H|_T)  \mathbf{q}_{T,j}.
\end{equation}
Again, the operator $\Cor$ depends on the uncertainty parameter
$\omega$.
Define the piecewise constant matrix field
$\mathcal A_H\in L^2(\Omega;P_0(\T_H\times\T_H;\mathbb{R}^{d\times d}))$
over $\T_H\times\T_H$,
for $T,K\in \T_H$ by
\begin{equation}\label{e:AHdef}
(\mathcal A_H|_{T,K})_{jk}
:= 
\frac{1}{|T|\,|K|}
\left(\delta_{T,K}\int_T \mathbf{A}_{jk}\,dx -
 e_{j}\cdot\int_K \mathbf{A}\nabla \mathbf{q}_{T,k}\,dx 
\right)
\end{equation}
($j,k=1,\dots,d$)
where $\delta$ is the Kronecker symbol.
The bilinear form 
$\mathfrak{a}: V\times V\to L^2(\Omega;\mathbb R)$ is given by
$$
\mathfrak{a}(v_H,z_H):=
\int_D\int_D 
   \nabla v_H(x) \cdot ( \mathcal A_H(x,y)\nabla z_H(y))\,dy\,dx 
   \quad\text{for any } v_H,z_H\in V_H.
$$
As pointed out in \cite{GallistlPeterseim2017}, 
there holds for all finite element functions $v_H,z_H\in V_H$ that
\begin{equation}\label{e:frakaeq}
\int_D \nabla v_H\cdot (\mathbf{A}\nabla(1-\Cor) z_H)\,dx
=
\mathfrak{a}(v_H,z_H).
\end{equation}

\begin{remark}
The nonlocal operator $\mathcal A_H$ is sparse in the sense that
$\mathcal A_H|_{T,K}$ equals zero whenever 
$\operatorname{dist}(T,K)\gtrsim \ell H$ 
for $T,K\in\T_H$.
It is therefore referred to as quasilocal.
\end{remark}

\begin{remark}
The left-hand side of \eqref{e:frakaeq} corresponds to 
a Petrov-Galerkin method with finite element trial functions and
modified test functions. Such multiscale basis functions were
proposed in \cite{MalqvistPeterseim2014}.
For averaging procedures over the stochastic variable, it will
turn out that the representation from the right-hand side
of \eqref{e:frakaeq} is preferable. In other words, we average
the nonlocal integral kernel rather than multiscale basis functions.
Therefore we employ the variant from \cite{GallistlPeterseim2017}
where in the discretization the right-hand side of the PDE is
only tested with standard finite element functions,
while in the original method \cite{MalqvistPeterseim2014} the
right hand side was tested with multiscale test functions.
Those would be random variables in our case. 
\end{remark}

If the oversampling parameter $\ell$ is chosen in the order of
magnitude $\mathcal{O}(\lvert\log H\rvert)$, it can be shown
(see e.g., \cite[proof of Prop.~6]{GallistlPeterseim2017})
that the bilinear form $\mathfrak{a}$ is coercive and continuous
\begin{equation}\label{e:fracacoerc}
\|\nabla v_H\|_{L^2(D)}
  \lesssim \mathfrak{a}(v_H,v_H) 
\lesssim \|\nabla v_H\|_{L^2(D)}
\quad\text{for all }v_H\in V_H
\end{equation}
for any $\omega\in\Omega$.
Hence, there exists a unique solution 
$\mathbf{u}_H\in L^2(\Omega;V_H)$ to
\begin{equation}\label{e:bfuHdef}
\mathfrak{a}(\mathbf{u}_H,v_H)
= (f,v_H)_{L^2(D)}
\quad\text{for all }v_H\in V_H.
\end{equation}

Is is known that the method of \cite{MalqvistPeterseim2014}
produces quasi-optimal results for every fixed $\omega$.
More precisely, for the variant considered here,
\cite[Prop.~1]{GallistlPeterseim2017} states
\begin{equation}\label{e:wcbaest}
\frac{\| \mathbf{u}(\omega)-\mathbf{u}_H(\omega)\|_{L^2(D)}}{\|f\|_{L^2(D)}}
\lesssim
H^2 + \wcba(\mathbf{A}(\omega),\T_H).
\end{equation}
The term $\wcba(\mathbf{A}(\omega),\T_H) )$ denotes
the worst-case best-approximation error
\begin{equation}\label{e:wcbadef}
\wcba(\mathbf{A}(\omega),\T_H) 
:=
\sup_{g\in L^2(D)\setminus\{0\}}
\inf_{v_H\in V_H} \frac{\|u(g,\mathbf{A}(\omega))- v_H\|_{L^2(D)}}{\|g\|_{L^2(D)}} 
\end{equation}
where for $g\in L^2(D)$, $u(g,\mathbf{A}(\omega))\in V$ solves 
the deterministic 
model problem with diffusion coefficient $\mathbf{A}(\omega)$
and right-hand side $g$.
In particular, the right-hand side of \eqref{e:wcbaest} is always
controlled by $H \|f\|_{L^2(D)}$.

The approximation by a deterministic model is based on the 
averaged integral kernel
$\bar{\mathcal A}_H :=\mathbb{E}[ \mathcal{A}_H]$.
In view of \eqref{e:AHdef},
the values of the piecewise constant integral kernel $\bar{\mathcal A}_H$
on two simplices $T,K\in\T_H$ are given by
\begin{equation}\label{e:barAHdef}
(\bar{\mathcal A}_H|_{T,K})_{jk}
= 
\frac{1}{|T|\,|K|}
\left(\delta_{T,K}\int_T \mathbb{E}[\mathbf{A}_{jk}]\,dx -
 e_{j}\cdot\int_K \mathbb{E}[\mathbf{A}\nabla \mathbf{q}_{T,k}]\,dx  
\right) .
\end{equation}
The corresponding bilinear form
$\bar{\mathfrak{a}}(\cdot,\cdot)$
given by
$$
\bar{\mathfrak{a}}(v_H,z_H):=
\int_D\int_D 
   \nabla v_H(x) \cdot (\bar{ \mathcal A}_H(x,y)\nabla z_H(y))\,dy\,dx 
   \quad\text{for any } v_H,z_H\in V_H.
$$ 
The discrete solution $u_H\in V_H$ to the quasilocal deterministic
model is given by
\begin{equation}\label{e:uHdef}
\bar{\mathfrak{a}}(u_H,v_H)
= (f,v_H)_{L^2(D)}
\quad\text{for all }v_H\in V_H.
\end{equation}

\begin{remark}
In practice, the stochastic averages in \eqref{e:barAHdef} 
are approximated through sampling procedures. This is indeed feasible
because for some $T\in\T_H$ and $\omega\in\Omega$, the computation of 
$\mathcal{A}_H|_{T,K}(\omega)$ for any $K\in\T_H$ corresponds to 
the solution of problem \eqref{e:qTelementcorrELL}, which is posed
on the quasilocal neighbourhood $D_T$ of $T$.
\end{remark}

\subsection{Compression to a  local deterministic coefficient}
Given the quasilocal upscaled coefficient, one may ask whether there
exists a suitable approximation by a PDE model.
In order to provide a fully local model, a further 
compression step is introduced \cite{GallistlPeterseim2017}.
The nonlocal bilinear form
$\mathfrak{a}(\cdot,\cdot)$ is approximated 
by a quadrature-like procedure as follows.
Define the piecewise constant coefficient
\begin{equation*}
\mathbf{A}_H\in L^2(\Omega;P_0(\T_H;\mathbb R^{d\times d}))
\quad\text{by}\quad
\mathbf{A}_H|_T:=
\sum_{K\in\T_H} |K|\, \mathcal{A}_H|_{T,K}.
\end{equation*}
For fixed $\omega\in\Omega$,
the tensor field  $\mathbf{A}_H(\omega)$ 
is the local effective diffusion coefficient
of \cite{GallistlPeterseim2017} on the mesh $\T_H$.
In particular,
$\mathbf{A}_H$ still depends on $x$ and $\omega$.
We define the deterministic diffusion tensor by
\begin{equation*}
 \bar{A}_H := \mathbb{E}[\mathbf{A}_H] .
\end{equation*}
By linearity of the expectation operator,
$\bar{A}_H$ is equivalently obtained by compressing the 
averaged operator $\bar{\mathcal A}_H$.

It is not guaranteed a~priori that $\bar A_H$ is uniformly
positive definite.
In what follows we therefore assume that 
$\bar{A}_H\in \mathcal{M}(D,\alpha/2,2\beta)$.
This condition can be checked a~posteriori.
We denote by $\uHt\in V_H$ the solution to the following
finite element system
\begin{equation}\label{e:uHtdef}
 \int_D \nabla \uHt \cdot (\bar{A}_H\nabla v_H)\,dx
  = (f,v_H)_{L^2(D)}
 \quad\text{for all } v_H\in V_H.
\end{equation}
This effective equation is the discretization of a PDE.
As described in Subsection~\ref{ss:analysisfullylocal},
the coefficient $\bar{A}_H$ can be regularized to some
$\bar{A}_H\smooth$ that leads to comparable accuracy.

\section{Error analysis}\label{s:erroranalysis}
This section provides $L^2$ error estimates for the 
upscaling schemes. The estimates combine 
a~priori and a~posteriori terms.
The measure for quantifying the error is the 
$L^2(\Omega;L^2(D))$ norm, denoted by
\begin{equation*}
||| \mathbf{v} |||
:=
\sqrt{\mathbb{E}[ \| \mathbf{v} \|_{L^2(D)}^2] }.
\end{equation*}
We will also provide error estimates for the $L^2$ norm
of the expected error.

\subsection{Error estimate for the quasilocal method}
\begin{definition}[model error estimator]\label{d:modelest}
For any $T\in\T_H$, denote
\begin{equation*}X(T) :=
\max_{\substack{K\in\T_H\\ K\cap\nei^\ell(T)\neq\emptyset}}
  \lvert T\rvert 
   \, 
   \Big\lvert
    \mathcal A_H|_{T,K}-\bar{ \mathcal A}_H|_{T,K}
   \Big\rvert  .
\end{equation*}
The model error estimator $\gamma$ is defined by
\begin{equation*}
\gamma:=\max_{T\in\T_H} \Big( \sqrt{\mathbb{E}[X(T)^2]}\Big)
\Bigg/
\Big(\max_{T\in\T_H}\max_{K\in\T_H} 
    \mathbb{E}[\big\lvert\bar{ \mathcal A}_H|_{T,K}
   \big\rvert ] \Big).
\end{equation*}
\end{definition}

\begin{remark}[normalization of $\gamma$]\label{r:gammascaling}
Throughout the analysis of this paper, the constants hidden in the 
notation $\lesssim$ may involve the contrast $\beta/\alpha$. 
We 
use
the scaling of $\gamma$ as in Definition~\ref{d:modelest}.
\end{remark}

The random variable $X$ measures local fluctuations of 
$\mathcal A_H$. Its expectation determines the model error estimator 
$\gamma$ that is part of the upper bound in the subsequent error 
estimate. It is a term to be computed a~posteriori. 

\begin{remark}
Note that we have not assumed any particular structure of the 
coefficient $\mathbf{A}$. Information on the validity of the 
discrete model is instead extracted from the a~posteriori model
error estimator $\gamma$. It is expected that small values of
$\gamma$ require a certain scale separation in the stochastic variable
in the sense of the CLT scaling.
If, for example, $\mathbf{A}$ is i.i.d.\ over $\T_\varepsilon$
(with $\T_\varepsilon$ from Subsection~\ref{ss:Teps}),
then the value of $\gamma$ is basically determined by the ratio
$(\varepsilon/H)^{d/2}$.
\end{remark}

\begin{lemma}\label{l:graduHdiff}
Let $\ell\approx \lvert\log H\rvert$.
Let $\mathbf{u}_H$ solve \eqref{e:bfuHdef} and let
$u_H$ solve \eqref{e:uHdef} with right-hand side
$f\in L^2(D)$.
Then, for $\rho:=\lvert\log H\rvert$,
\begin{align*}
|||\nabla (\mathbf{u}_H-u_H)|||
\lesssim 
 \rho^d
 \gamma
 \|f\|_{L^2(D)}
\end{align*}
for the model error estimator $\gamma$ from
Definition~\ref{d:modelest}.
\end{lemma}
\begin{proof}
Denote 
$\mathbf{e}_H:=\mathbf{u}_H-u_H$.
The coercivity 
\eqref{e:fracacoerc} of the multiscale bilinear form for any atom
$\omega\in\Omega$ and the representation as integral operator
reveal
\begin{equation*}
\begin{aligned}
 \|\nabla \mathbf{e}_H\|_{L^2(D)}^2
 \lesssim
 \mathfrak{a}(\mathbf{e}_H,\mathbf{e}_H)
 =
 \int_D\int_D 
   \nabla (\mathbf{u}_H(x)-u_H(x)) \cdot 
             ( \mathcal A_H(x,y)\nabla \mathbf{e}_H(y))\,dy\,dx .
\end{aligned}
\end{equation*}
Abbreviate
$\mathcal{E}|_{T,K}
   :=
    \bar{ \mathcal A}_H|_{T,K} - \mathcal A_H|_{T,K}$.
Adding and subtracting $\bar{ \mathcal A}_H(x,y)$
together with the discrete solution properties of $\mathbf{u}_H$ and
$u_H$ lead to
\begin{equation*}
\begin{aligned}
 \|\nabla \mathbf{e}_H\|_{L^2(D)}^2
 &
 \lesssim
 \int_D\int_D 
   \nabla u_H(x) \cdot 
     (\bar{ \mathcal A}_H(x,y)-(\mathcal A_H(x,y))  
                       \nabla \mathbf{e}_H(y))\,dy\,dx
 \\
 &
 =
 \sum_{T\in\T_H}\sum_{\substack{K\in\T_H\\ K\cap\nei^\ell(T)\neq\emptyset}}
  \lvert T\rvert \lvert K\rvert 
    \; \nabla u_H|_T \cdot (\mathcal{E}|_{T,K} \nabla \mathbf{e}_H|_K)
\end{aligned}
\end{equation*}
where it was used that $\nabla u_H$ and $\nabla \mathbf{e}_H$ are piecewise
constant.
For any fixed $T\in\T_H$, the shape regularity of the mesh
and equivalence of norms in the finite-dimensional space
$\mathbb R^N$ with $N=\mathcal{O}(\ell^d)$ lead to
\begin{equation*}
\begin{aligned}
&
\sum_{\substack{K\in\T_H\\ K\cap\nei^\ell(T)\neq\emptyset}}
  \lvert T\rvert \lvert K\rvert 
    \; \nabla u_H|_T \cdot (\mathcal{E}|_{T,K} \nabla \mathbf{e}_H|_K)
\\
&\qquad
\lesssim
X(T) \lvert T\rvert^{1/2} \lvert \nabla u_H|_T\rvert
    \sum_{\substack{K\in\T_H\\ K\cap\nei^\ell(T)\neq\emptyset}}
     \lvert K\rvert^{1/2} \lvert\nabla \mathbf{e}_H|_K\rvert
\\
&\qquad\lesssim
\ell^{d/2}
   X(T) \|\nabla u_H\|_{L^2(T)} \|\nabla \mathbf{e}_H\|_{L^2(\nei^\ell(T))} .
\end{aligned}
\end{equation*}
The combination of the foregoing two displayed estimates with
the Cauchy inequality and the finite overlap of the patch
domains $\nei^\ell(T)$ containing 
$\mathcal{O}(\ell^d)$ elements
therefore proves
\begin{equation*}
\begin{aligned}
 \|\nabla \mathbf{e}_H\|_{L^2(D)}^2
 &
 \lesssim
 \ell^d
 \sqrt{\sum_{T\in\T_H} X(T)^2 \|\nabla u_H\|_{L^2(T)}^2}
 \|\nabla \mathbf{e}_H\|_{L^2(D)} .
\end{aligned}
\end{equation*}
After dividing by $\|\nabla \mathbf{e}_H\|_{L^2(D)}$,
taking squares and the
expectation, we arrive at
\begin{equation*}
\begin{aligned}
|||\nabla \mathbf{e}_H|||^2
 &
 \lesssim
 \ell^{2d}
 \mathbb{E}\left[\sum_{T\in\T_H}X(T)^2 \|\nabla u_H\|_{L^2(T)}^2
            \right] .
\end{aligned}
\end{equation*}
This and the stability
of the discrete problem for $u_H$ prove
\begin{equation*}
\begin{aligned}
|||\nabla \mathbf{e}_H|||^2
 &
 \lesssim
 \ell^{2d}
 \max_{T\in\T_H} \Big( \mathbb{E}[X(T)^2]\Big)
 \|f\|_{L^2(D)}^2
\lesssim 
 \ell^{2d}
 \gamma^2
 \|f\|_{L^2(D)}^2.
\end{aligned}
\end{equation*}
This concludes the proof.
\end{proof}

\begin{proposition}[error estimate for the quasilocal method]
       \label{p:errorestimateQlocal}
Let $\ell\approx \lvert\log H\rvert$.
Let $\mathbf{u}$ solve \eqref{e:diff1drandweak} and let
$u_H$ solve \eqref{e:uHdef} with right-hand side
$f\in L^2(D)$.
Then, for $\rho:=\lvert\log H\rvert$,
\begin{equation}\label{e:errEstQuasiEq1}
\begin{aligned}
||| \mathbf{u}-u_H |||
&
\lesssim 
(
H^2 + \mathbb{E}[\wcba(\mathbf{A},\T_H)]
+
\rho^d  \gamma )
 \|f\|_{L^2(D)} 
\\
&
\lesssim 
(H + \rho^d  \gamma )  \|f\|_{L^2(D)}
\end{aligned}
\end{equation}
for the model error estimator $\gamma$ from
Definition~\ref{d:modelest}.
Furthermore, the following higher-order error estimate holds for the 
norm of the expected error
\begin{equation}\label{e:errEstQuasiEq2}
\begin{aligned}
\| \mathbb{E}[\mathbf{u}]-u_H \|_{L^2(D)}
&
\lesssim 
(
H^2 + \mathbb{E}[\wcba(\mathbf{A},\T_H)]
+
\rho^{2d}  \gamma^2 )
 \|f\|_{L^2(D)} .
\end{aligned}
\end{equation}
\end{proposition}
\begin{proof}
Recall that $\mathbf{u}_H$ denotes the solution to
\eqref{e:bfuHdef}.
We
start
from the triangle inequality
\begin{equation}\label{e:triIneqErr}
||| \mathbf{u}-u_H |||
\leq 
||| \mathbf{u}-\mathbf{u}_H |||
+
||| \mathbf{u}_H- u_H  |||.
\end{equation}
The first term on the right-hand side of \eqref{e:triIneqErr}
is bounded with the estimate
\eqref{e:wcbaest}
\begin{equation}\label{e:errEstTerm1}
\begin{aligned}
||| \mathbf{u}-\mathbf{u}_H |||
\lesssim
(H^2 +  \mathbb{E}[\wcba(\mathbf{A},\T_H)] )\|f\|_{L^2(D)}.
\end{aligned}
\end{equation}
The second term on the right-hand side of \eqref{e:triIneqErr}
is controlled through Friedrichs' inequality and 
Lemma~\ref{l:graduHdiff}, so that we obtain the first stated
estimate of \eqref{e:errEstQuasiEq1}.
The second follows from the observation that
$\wcba(\mathbf{A},\T_H) \lesssim H$.

For the proof of \eqref{e:errEstQuasiEq2}, we employ a duality
argument.
Denote 
$\mathbf{e}_H:=\mathbf{u}_H-u_H$ and 
let $\mathbf{z}_H\in L^2(\Omega;V)$ solve
\begin{equation}\label{e:bfzHdef}
\mathfrak{a} (v_H,\mathbf{z}_H) = (\mathbb{E}[\mathbf{e}_H],v_H)_{L^2(D)}
\quad\text{for all } v_H\in V_H
\quad\mathbb{P}\text{-a.s}.
\end{equation}
Let $z_H\in V_H$ denote the solution to
\begin{equation}\label{e:zHdef}
\bar{\mathfrak{a}} (v_H,z_H) = (\mathbb{E}[\mathbf{e}_H],v_H)_{L^2(D)}
\quad\text{for all } v_H\in V_H.
\end{equation}
Then, \eqref{e:bfzHdef} implies
\begin{equation*}
\|\mathbb{E}[\mathbf{e}_H]\|_{L^2(D)}^2
=
\mathbb{E}[(\mathbb{E}[\mathbf{e}_H],\mathbf{e}_H)_{L^2(D)}]
=
\mathbb{E}[\mathfrak{a} (\mathbf{u}_H-u_H,\mathbf{z}_H)].
\end{equation*}
Furthermore, \eqref{e:bfuHdef},
the definition of $\bar{\mathfrak{a}}$ and \eqref{e:uHdef} lead to the Galerkin orthogonality
\begin{equation*}
\mathbb{E}[\mathfrak{a} (\mathbf{u}_H-u_H,z_H)]
= (f,z_H)_{L^2(D)} - \mathbb{E}[\mathfrak{a} (u_H,z_H)]
= (f,z_H)_{L^2(D)} - \bar{\mathfrak{a}} (u_H,z_H)
=0.
\end{equation*}
Thus,
\begin{equation*}
\begin{aligned}
\|\mathbb{E}[\mathbf{e}_H]\|_{L^2(D)}^2
&=
\mathbb{E}[\mathfrak{a} (\mathbf{u}_H-u_H,\mathbf{z}_H-z_H)]
\lesssim
||| \nabla(\mathbf{u}_H-u_H)|||\,|||\nabla(\mathbf{z}_H-z_H)|||.
\end{aligned}
\end{equation*}
Each of the terms on the right-hand side can be bounded with
help of Lemma~\ref{l:graduHdiff}
because
\eqref{e:bfzHdef} and \eqref{e:zHdef} correspond to
\eqref{e:bfuHdef} and \eqref{e:uHdef} where the right-hand side
$f$ is replaced by $\mathbb{E}[\mathbf{e}_H]$.
Therefore,
\begin{equation*}
\|\mathbb{E}[\mathbf{e}_H]\|_{L^2(D)}^2
\lesssim
 \rho^{2d} \gamma^2 \|f\|_{L^2(D)} \|\mathbb{E}[\mathbf{e}_H]\|_{L^2(D)}.
\end{equation*}
This proves
$\|\mathbb{E}[\mathbf{e}_H]\|_{L^2(D)}
\lesssim  \rho^{2d} \gamma^2 \|f\|_{L^2(D)}$.
In order to conclude the proof of \eqref{e:errEstQuasiEq2},
we use the triangle inequality
\begin{equation*}
\|\mathbb{E}[\mathbf{u}]-u_H\|_{L^2(D)}
\leq
\|\mathbb{E}[\mathbf{u}-\mathbf{u}_H]\|_{L^2(D)}
+
\|\mathbb{E}[\mathbf{e}_H]\|_{L^2(D)}
\end{equation*}
and observe that Jensen's inequality implies
$\|\mathbb{E}[\mathbf{u}-\mathbf{u}_H]\|_{L^2(D)}
\leq ||| \mathbf{u}-\mathbf{u}_H ||| $.
Altogether
\begin{equation*}
\|\mathbb{E}[\mathbf{u}]-u_H\|_{L^2(D)}
\lesssim
||| \mathbf{u}-\mathbf{u}_H |||
+
\rho^{2d} \gamma^2 \|f\|_{L^2(D)} 
\end{equation*}
and the combination with \eqref{e:errEstTerm1} 
implies \eqref{e:errEstQuasiEq2}.
\end{proof}

\subsection{Error estimate for the fully local method}
\label{ss:analysisfullylocal}
While the quasilocal method admits an error estimate under
mild regularity assumptions on the solution,
the error estimate for the fully local method is restricted to
the planar case and provides sublinear rates depending on the 
$W^{s,q}$ regularity of the solution to the deterministic model
problem with some regularized coefficient.
More precisely, it was shown in
\cite[Lemma~7]{GallistlPeterseim2017} that, given $\bar{A}_H$,
there exists a regularized coefficient
$\bar{A}_H\smooth\in W^{1,\infty}(D;\mathbb R^{d\times d})$
such that
1) The piecewise integral mean is conserved, i.e.,
\begin{equation*}
\int_T \bar{A}_H\smooth \,dx = \int_T \bar{A}_H\,dx
\quad\text{for all }T\in \T_H.
\end{equation*}
2) The eigenvalues of $\sym(\bar{A}_H\smooth)$ lie in the interval
$[\alpha/4,4\beta]$.
3) The derivative satisfies the bound
\begin{equation*}
\|\nabla \bar{A}_H\smooth\|_{L^\infty(D)}
\leq C
\eta(\bar{A}_H)
\end{equation*}
for some constant $C$ that depends on the shape-regularity of
$\T_H$
and for the expression
\begin{equation}\label{e:etadef}
\eta(\bar{A}_H) :=
H^{-1} \| [\bar{A}_H] \|_{L^\infty(\mathcal F_H)}
\big(1+\alpha^{-1}\| [\bar{A}_H] \|_{L^\infty(\mathcal F_H)}\big)
\Big/ \big((\alpha+\beta)/2\big) .
\end{equation}
Here $[\cdot]$ defines the inter-element jump and
$\mathcal F_H$ denotes the set of
interior hyper-faces of $\T_H$.
The coefficients
$\bar{A}_H$ and $\bar{A}_H\smooth$ lead to the same finite element
solution.
Let $u\smooth\in V$ denote the solution
\begin{equation}\label{e:PDEsmooth}
\int_\Omega\nabla {u\smooth}\cdot ( \bar{A}_H\smooth \nabla v) \,dx =
 F(v)\quad\text{for all } v\in V.
\end{equation}
In particular,  the integral conservation property
for $\bar{A}_H\smooth$ stated above implies that
$\uHt$ is the finite element approximation to
$u\smooth$.
The solution $u\smooth$ with respect to the 
regularized coefficient $\bar{A}_H\smooth$ serves to quantify smoothness
in terms of elliptic regularity.

\begin{proposition}[error estimate for the fully local method]
       \label{p:errorestimateLocal}
Assume that the solution $u\smooth$ to
\eqref{e:PDEsmooth} belongs to $H^{1+s}(D)$
for some $0<s\leq 1$.
Then, for $f\in L^2(D)$
and $\rho:=\lvert\log H\rvert$,
\begin{equation*}
||| \mathbf{u}-\uHt |||
\lesssim 
H \|f\|_{L^2(D)}
+
\rho^d  \gamma  \|f\|_{L^2(D)} 
+
 H^s \rho^{s+d/2}
 \;\big(1+\eta(\bar{A}_H\loc)\big)^{s}
 \|f \|_{L^2(D)}
\end{equation*}
for $\gamma$ from Definition~\ref{d:modelest}. If the domain $D$ is convex, then $s$ can be chosen as $s=1$.
\end{proposition}
\begin{proof}
The triangle inequality reads
\begin{equation*}
|||\mathbf{u}-\uHt |||
\leq
||| \mathbf{u}-u_H |||
+\| u_H-\uHt\|_{L^2(D)} .
\end{equation*}
The first term is estimated via
Proposition~\ref{p:errorestimateQlocal} and the observation that
the term $\mathbb{E}[\wcba(\mathbf{A},\T_H)]$ is bounded
by some constant times $H$.
It remains to estimate the second, purely deterministic term.
The difference of $u_H$ and $\uHt$ was already estimated in
\cite[Proposition~8]{GallistlPeterseim2017} as
\begin{equation*}
\|\nabla(u_H-\uHt)\|_{L^2(D)}
\lesssim
H^s \lvert\log H\rvert^{s+d/2}
 \;\big(1+\eta(\bar{A}_H\loc)\big)^s
 \|f \|_{L^2(D)}.
\end{equation*}
This concludes the proof.
\end{proof}

\section{Numerical illustration}\label{s:num}
We consider the planar square domain $D=(0,1)^2$ with homogeneous
Dirichlet boundary and the right-hand side $f\equiv 1$.
The finite element meshes are uniform red refinements of the 
triangulation displayed in Figure~\ref{f:initialmesh}.
We adopt the setting of Subsection~\ref{ss:Teps} and 
the mesh $\T_\varepsilon$ has mesh-size
$\varepsilon=\{2^{-5},2^{-6},2^{-7}\}\sqrt{2}$. The coefficient is scalar 
i.i.d.\ and, on each cell of $\T_\varepsilon$, it is uniformly distributed
in the interval $[\alpha,\beta]=[1,10]$.
The fine-scale mesh for the solution of the corrector problems
and the reference solution $\mathbf{u}_h$ has mesh-size
$h=2^{-9}\sqrt{2}$. All expected values are replaced with suitable empirical means.

Figure~\ref{f:L2L2errorsiid}
displays the relative errors 
$||| \mathbf{u}-u_H|||$ and 
$|||\mathbf{u}-\tilde u_H|||$ 
for the solution $u_H\in V_H$ to the quasilocal effective model \eqref{e:uHdef}
and 
the solution $\tilde u_H\in V_H$ to the local effective model \eqref{e:uHtdef}
in the $L^2$-$L^2$ norm $|||\cdot|||$.
The two approximations are compared on a sequence of meshes with mesh
size
$H=\{2^{-2},   2^{-3},   2^{-4},   2^{-5}\}\sqrt{2}$.
We consider only errors with respect to 
the reference solution $\mathbf{u}_h$.
It is observed that the quasilocal method always leads to a smaller error
than the local method. For coarse meshes we observe a convergence rate
between $H$ and $H^2$.
For fine meshes with $H\lesssim \sqrt{\varepsilon}$, the approximation by
the quasilocal method does not improve with respect to the 
previous mesh. Our interpretation is that the stochastic
error dominates in this regime.
In terms of the error estimate of 
Proposition~\ref{p:errorestimateQlocal}
this means, that the 
term $\gamma$ (resp.\ $\gamma^2$)
on the right-hand side is larger than the error
that would be possible in a deterministic setting. 
For $\varepsilon=2^{-7}\sqrt{2}$,
the values of the model error estimators
$\gamma$ and $H\eta$ are displayed in Figure~\ref{f:estimatorsiid}.
The value of $\gamma$ was rescaled as suggested in
Remark~\ref{r:gammascaling}.
It is observed that its values scale as $\varepsilon/H$. This is what one would expect from the central limit theorem because, in the planar 
case, each coarse cell covers $\mathcal{O}((H/\varepsilon)^2)$
many cells in $\T_\varepsilon$.

Figure~\ref{f:L2experrorsiid}
displays the relative errors 
$\| \mathbb{E}[\mathbf{u}]_h-u_H\|_{L^2(D)}$ and 
$\| \mathbb{E}[\mathbf{u}]_h-\tilde u_H\|_{L^2(D)}$.
On coarse meshes, the convergence rate $H^2$ can be observed.
Again, for fine meshes with $H\gtrsim\sqrt{\varepsilon}$,
no improvement is achieved through mesh-refinement. Altogether, we observe that the methods perform well up to the 
critical regime $H\approx\sqrt{\varepsilon}$ in this two-dimensional
example where a pure (local of quasi-local) deterministic approximation is no longer sufficiently accurate.

\begin{figure}
\centering
\begin{tikzpicture}[scale=.7]
\draw[step=1.0,black,very thick] (0,0) grid (4,4);
\draw[thick] (2,0)--(4,2)--(2,4)--(0,2)--cycle;
\draw[thick]  (2,1)--(3,2)--(2,3)--(1,2)--cycle;
\draw[thick]  (0,1)--(1,0)
      (3,0)--(4,1)
      (4,3)--(3,4)
      (1,4)--(0,3);
\end{tikzpicture}
\caption{Initial mesh of size $H=2^{-2}\sqrt{2}$.
         \label{f:initialmesh}}
\end{figure}
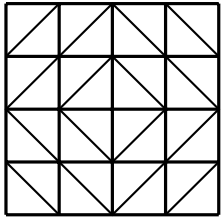

\begin{figure}[p]
\centering
\includegraphics{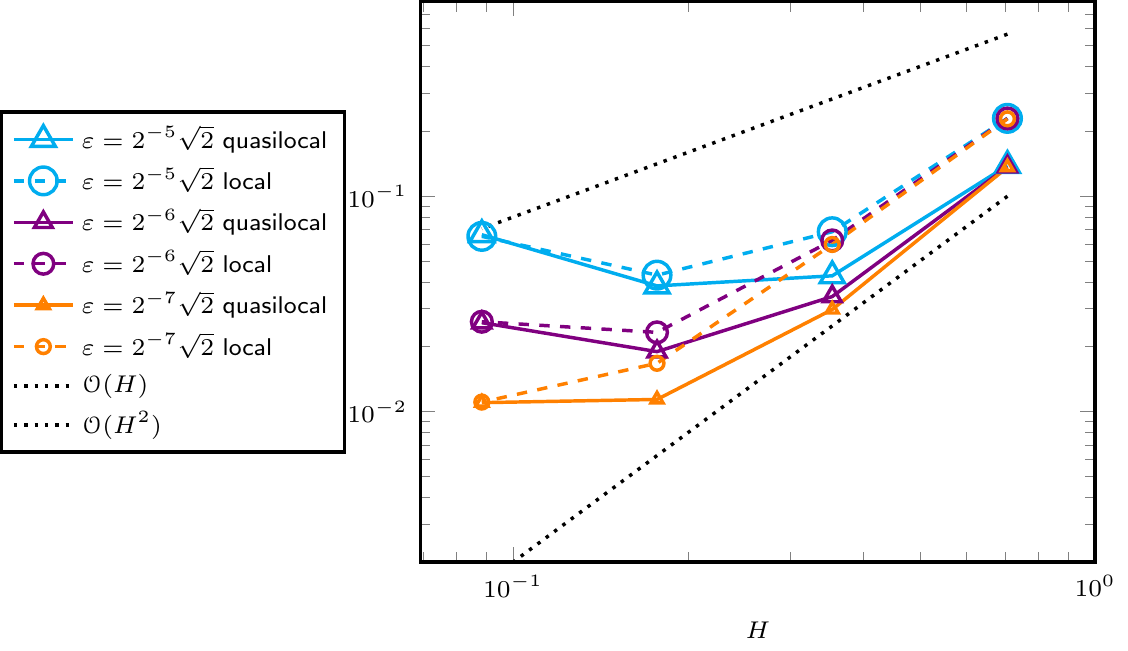}
\caption{Relative errors $||| \mathbf{u}-u_H|||$ (quasilocal)
         and $||| \mathbf{u}-\tilde u_H|||$ (local)
          in dependence of the coarse mesh size $H$.
         \label{f:L2L2errorsiid}}
\end{figure}

\begin{figure}
\centering
\includegraphics{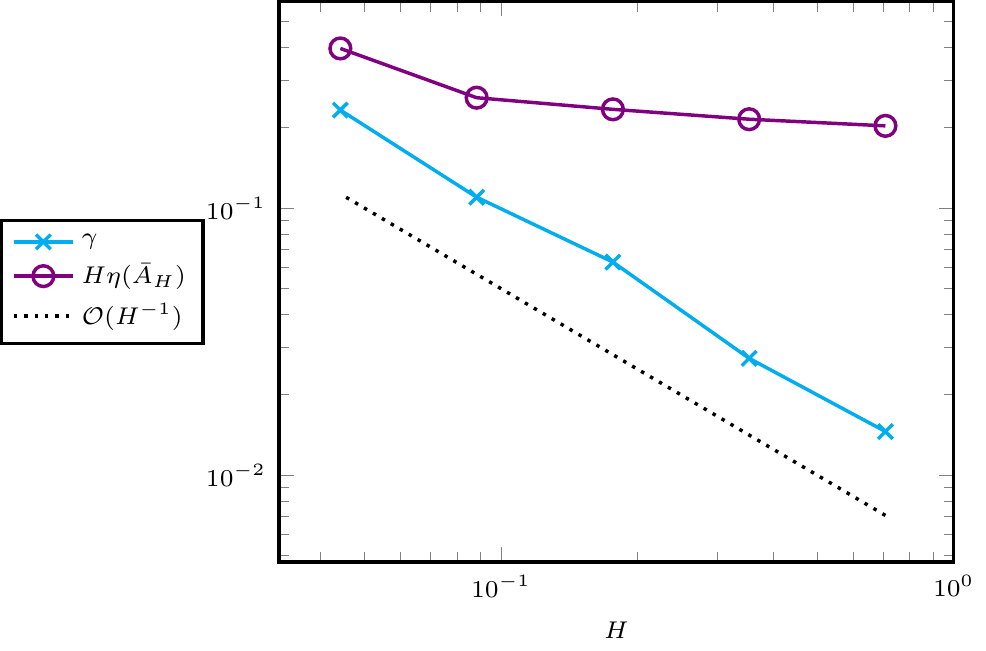}
\caption{Model error estimators $H$ for $\varepsilon=2^{-7}\sqrt{2}$.
         \label{f:estimatorsiid}}
\end{figure}

\begin{figure}[p]
\centering
\includegraphics{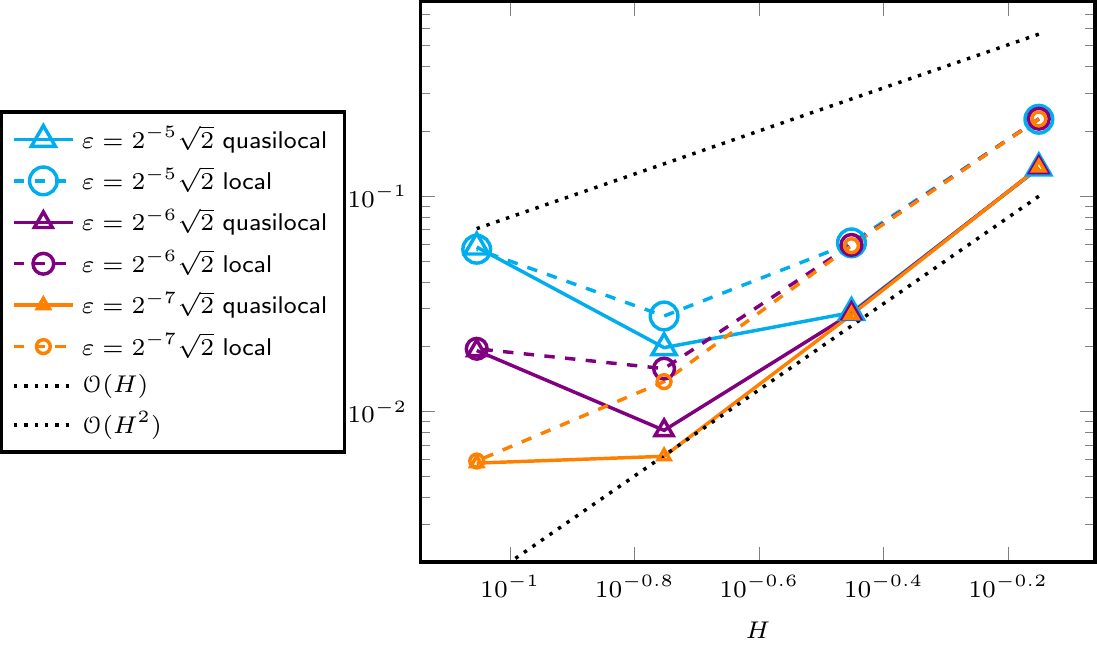}
\caption{Relative errors $\| \mathbb{E}[\mathbf{u}_h]-u_H\|_{L^2(D)}$ (quasilocal)
         and $\| \mathbb{E}[\mathbf{u}_h]-\tilde u_H\|_{L^2(D)}$ (local)
          in dependence of the coarse mesh size $H$.
         \label{f:L2experrorsiid}}
\end{figure}

\section{Conclusive comments}\label{s:conclusion}
The proposed numerical homogenization procedure approximates the 
stochastic coefficient by the expectation of a quasilocal effective
model. The design of intermediate stochastic models that carry more 
information on the stochastic dependence than a purely deterministic
coefficient will be considered in future work.
The presented error estimates are independent on any assumptions on
the uncertainty and contain an a~posteriori model error estimator
$\gamma$.
In the case that more structure on the coefficient is given,
we expect that an a~priori error estimate for $\gamma$ in terms
of $H$ and $\varepsilon$ can be derived.

{\footnotesize
\bibliographystyle{alpha}

\begin{thebibliography}{JMSDO14}

\bibitem[ACL{\etalchar{+}}12]{AnantharamanEtAl_Review2012}
A.~Anantharaman, R.~Costaouec, C.~{Le Bris}, F.~Legoll, and F.~Thomines.
\newblock Introduction to numerical stochastic homogenization and the related
  computational challenges: some recent developments.
\newblock In {\em Multiscale modeling and analysis for materials simulation},
  volume~22 of {\em Lect. Notes Ser. Inst. Math. Sci. Natl. Univ. Singap.},
  pages 197--272. World Sci. Publ., Hackensack, NJ, 2012.

\bibitem[AKM17]{ArmstrongKuusiMourrat2017}
S.~Armstrong, T.~Kuusi, and J.-C. Mourrat.
\newblock The additive structure of elliptic homogenization.
\newblock {\em Invent. Math.}, 208(3):999--1154, 2017.

\bibitem[All92]{MR1185639}
G.~Allaire.
\newblock Homogenization and two-scale convergence.
\newblock {\em SIAM J. Math. Anal.}, 23(6):1482--1518, 1992.

\bibitem[AS16]{ArmstrongSmart2016}
S.~N. Armstrong and C.~K. Smart.
\newblock Quantitative stochastic homogenization of convex integral
  functionals.
\newblock {\em Ann. Sci. \'Ec. Norm. Sup\'er. (4)}, 49(2):423--481, 2016.

\bibitem[BJ11]{BalJing2011}
G.~Bal and W.~Jing.
\newblock Corrector theory for {MSFEM} and {HMM} in random media.
\newblock {\em Multiscale Model. Simul.}, 9(4):1549--1587, 2011.

\bibitem[BLP78]{BLP78}
A.~Bensoussan, J.-L. Lions, and G.~Papanicolaou.
\newblock {\em Asymptotic Analysis for Periodic Structures}.
\newblock North-Holland Publ., 1978.

\bibitem[BP04]{BourgeatPiatnitski2004}
A.~Bourgeat and A.~Piatnitski.
\newblock Approximations of effective coefficients in stochastic
  homogenization.
\newblock {\em Ann. Inst. H. Poincar\'e Probab. Statist.}, 40(2):153--165,
  2004.

\bibitem[CS08]{ChenSavchuk2007}
Z.~Chen and T.~Y. Savchuk.
\newblock Analysis of the multiscale finite element method for nonlinear and
  random homogenization problems.
\newblock {\em SIAM J. Numer. Anal.}, 46(1):260--279, 2007/08.

\bibitem[DG75]{Gio75}
E.~De~Giorgi.
\newblock Sulla convergenza di alcune successioni d'integrali del tipo
  dell'area.
\newblock {\em Rend. Mat. (6)}, 8:277--294, 1975.
\newblock Collection of articles dedicated to Mauro Picone on the occasion of
  his ninetieth birthday.

\bibitem[DGO16]{DuerinckxGloriaOtto2016}
M.~Duerinckx, A.~Gloria, and F.~Otto.
\newblock The structure of fluctuations in stochastic homogenization.
\newblock {\em arXiv e-prints}, 1602.01717 [math.AP], 2016.

\bibitem[EMZ05]{EMingZhang2005}
W.~E, P.~Ming, and P.~Zhang.
\newblock Analysis of the heterogeneous multiscale method for elliptic
  homogenization problems.
\newblock {\em J. Amer. Math. Soc.}, 18(1):121--156, 2005.

\bibitem[GH16]{GloriaHabibi2016}
A.~Gloria and Z.~Habibi.
\newblock Reduction in the resonance error in numerical homogenization {II}:
  {C}orrectors and extrapolation.
\newblock {\em Found. Comput. Math.}, 16(1):217--296, 2016.

\bibitem[Glo12a]{Gloria2012m2an}
A.~Gloria.
\newblock Numerical approximation of effective coefficients in stochastic
  homogenization of discrete elliptic equations.
\newblock {\em ESAIM Math. Model. Numer. Anal.}, 46(1):1--38, 2012.

\bibitem[Glo12b]{Gloria2012proc}
A.~Gloria.
\newblock Numerical homogenization: survey, new results, and perspectives.
\newblock In {\em Mathematical and numerical approaches for multiscale
  problem}, volume~37 of {\em ESAIM Proc.}, pages 50--116. EDP Sci., Les Ulis,
  2012.

\bibitem[GMP10]{GintingMalqvistPresho2010}
V.~Ginting, A.~M{\aa}lqvist, and M.~Presho.
\newblock A novel method for solving multiscale elliptic problems with randomly
  perturbed data.
\newblock {\em Multiscale Model. Simul.}, 8(3):977--996, 2010.

\bibitem[GNO14]{GloriaNeukammOtto2014}
A.~Gloria, S.~Neukamm, and F.~Otto.
\newblock A regularity theory for random elliptic operators.
\newblock {\em arXiv e-prints}, 1409.2678, 2014.

\bibitem[GNO15]{GloriaNeukammOtto2015}
A.~Gloria, S.~Neukamm, and F.~Otto.
\newblock Quantification of ergodicity in stochastic homogenization: optimal
  bounds via spectral gap on {G}lauber dynamics.
\newblock {\em Invent. Math.}, 199(2):455--515, 2015.

\bibitem[GO11]{GloriaOtto2011}
A.~Gloria and F.~Otto.
\newblock An optimal variance estimate in stochastic homogenization of discrete
  elliptic equations.
\newblock {\em Ann. Probab.}, 39(3):779--856, 2011.

\bibitem[GO12]{GloriaOtto2012}
A.~Gloria and F.~Otto.
\newblock An optimal error estimate in stochastic homogenization of discrete
  elliptic equations.
\newblock {\em Ann. Appl. Probab.}, 22(1):1--28, 2012.

\bibitem[GO15]{GloriaOtto2015}
A.~Gloria and F.~Otto.
\newblock The corrector in stochastic homogenization: optimal rates, stochastic
  integrability, and fluctuations.
\newblock {\em arXiv e-prints}, 1510.08290 [math.AP], 2015.

\bibitem[GO17]{GloriaOtto2017}
A.~Gloria and F.~Otto.
\newblock Quantitative results on the corrector equation in stochastic
  homogenization.
\newblock {\em J. Eur. Math. Soc. (JEMS)}, 19(11):3489--3548, 2017.

\bibitem[GP15]{GallistlPeterseim2015}
D.~{Gallistl} and D.~{Peterseim}.
\newblock Stable multiscale {P}etrov-{G}alerkin finite element method for high
  frequency acoustic scattering.
\newblock {\em Comput. Methods Appl. Mech. Eng.}, 295:1--17, 2015.

\bibitem[GP17]{GallistlPeterseim2017}
D.~Gallistl and D.~Peterseim.
\newblock Computation of quasilocal effective diffusion tensors and connections
  to the mathematical theory of homogenization.
\newblock {\em Multiscale Model. Simul.}, 15(4), 2017.

\bibitem[HP13]{HenningPeterseim2013}
P.~Henning and D.~Peterseim.
\newblock Oversampling for the multiscale finite element method.
\newblock {\em Multiscale Model. Simul.}, 11(4):1149--1175, 2013.

\bibitem[JKO94]{JKO94}
V.~V. Jikov, S.~M. Kozlov, and O.~A. Oleinik.
\newblock {\em Homogenization of Differential Operators and Integral
  Functionals}.
\newblock Springer-Verlag, 1994.

\bibitem[JMSDO14]{JagalurSahniDoostanOberai2014}
J.~Jagalur~Mohan, O.~Sahni, A.~Doostan, and A.~A. Oberai.
\newblock Variational multiscale analysis: the fine-scale {G}reen's function
  for stochastic partial differential equations.
\newblock {\em SIAM/ASA J. Uncertain. Quantif.}, 2(1):397--422, 2014.

\bibitem[Koz79]{Kozlov1979}
S.~M. Kozlov.
\newblock The averaging of random operators.
\newblock {\em Mat. Sb. (N.S.)}, 109(151)(2):188--202, 327, 1979.

\bibitem[KPY18]{Kornhuber.Peterseim.Yserentant:2016}
R.~{Kornhuber}, D.~{Peterseim}, and H.~{Yserentant}.
\newblock {An analysis of a class of variational multiscale methods based on
  subspace decomposition}.
\newblock {\em Math. Comp.}, 87:2765--2774, 2018.

\bibitem[KY16]{KornhuberYserentant2015}
R.~Kornhuber and H.~Yserentant.
\newblock Numerical homogenization of elliptic multiscale problems by subspace
  decomposition.
\newblock {\em Multiscale Modeling \& Simulation}, 14(3):1017--1036, 2016.

\bibitem[LLT14]{LeBrisLegollThomines2014}
C.~{Le Bris}, F.~Legoll, and F.~Thomines.
\newblock Multiscale finite element approach for ``weakly'' random problems and
  related issues.
\newblock {\em ESAIM Math. Model. Numer. Anal.}, 48(3):815--858, 2014.

\bibitem[MP14]{MalqvistPeterseim2014}
A.~M{\aa}lqvist and D.~Peterseim.
\newblock Localization of elliptic multiscale problems.
\newblock {\em Math. Comp.}, 83(290):2583--2603, 2014.

\bibitem[MT78]{Mur78}
F.~Murat and L.~Tartar.
\newblock H-convergence.
\newblock {\em S\'{e}minaire d'Analyse Fonctionnelle et Num\'{e}rique de
  l'Universit\'{e} d'Alger}, 1978.

\bibitem[Ngu89]{MR990867}
G.~Nguetseng.
\newblock A general convergence result for a functional related to the theory
  of homogenization.
\newblock {\em SIAM J. Math. Anal.}, 20(3):608--623, 1989.

\bibitem[Pet16]{Peterseim:2015}
D.~Peterseim.
\newblock Variational multiscale stabilization and the exponential decay of
  fine-scale correctors.
\newblock In G.~R. Barrenechea, F.~Brezzi, A.~Cangiani, and E.~H. Georgoulis,
  editors, {\em Building Bridges: Connections and Challenges in Modern
  Approaches to Numerical Partial Differential Equations}, volume 114 of {\em
  Lect. Notes Comput. Sci. Eng.}, pages 341--367. Springer, 2016.

\bibitem[PV81]{PapanicolaouVaradhan1981}
G.~C. Papanicolaou and S.~R.~S. Varadhan.
\newblock Boundary value problems with rapidly oscillating random coefficients.
\newblock In {\em Random fields, {V}ol. {I}, {II} ({E}sztergom, 1979)},
  volume~27 of {\em Colloq. Math. Soc. J\'anos Bolyai}, pages 835--873.
  North-Holland, Amsterdam-New York, 1981.

\bibitem[Spa68]{Spagnolo:1968}
S.~Spagnolo.
\newblock Sulla convergenza di soluzioni di equazioni paraboliche ed
  ellittiche.
\newblock {\em Ann. Scuola Norm. Sup. Pisa (3) 22 (1968), 571-597; errata,
  ibid. (3)}, 22:673, 1968.

\bibitem[Yur86]{Yurinskii1986}
V.~V. Yurinski{\u\i}.
\newblock Averaging of symmetric diffusion in a random medium.
\newblock {\em Sibirsk. Mat. Zh.}, 27(4):167--180, 215, 1986.

\bibitem[ZCH15]{ZhangCiHou2015}
Z.~Zhang, M.~Ci, and T.~Y. Hou.
\newblock A multiscale data-driven stochastic method for elliptic {PDE}s with
  random coefficients.
\newblock {\em Multiscale Model. Simul.}, 13(1):173--204, 2015.

\end{thebibliography}

\newcommand{\etalchar}[1]{$^{#1}$}

}

\section*{Acknowledgment}
D.~Gallistl acknowledges support by the 
Deutsche Forschungsgemeinschaft (DFG) through CRC 1173.
Main parts of this paper were written while the authors enjoyed
the kind hospitality of the Hausdorff Institute for Mathematics in Bonn during the trimester program on multiscale problems.
\end{document}